\documentclass[11pt]{article}
\usepackage[utf8]{inputenc}
\usepackage{authblk}
\usepackage{mathtools}
\usepackage[english]{babel}
\usepackage{amsmath}
\usepackage{amsthm}
\usepackage{amsfonts}
\usepackage{amssymb}
\usepackage{graphicx}
\usepackage[usenames,dvipsnames]{color}
\usepackage{hyperref}
\usepackage[left=2cm,right=2cm,top=2cm,bottom=2cm]{geometry}
\usepackage{dsfont}
\usepackage{pgf,tikz,pgfplots}
\usepackage{enumerate}

\usepackage[normalem]{ulem}
\usepackage[T1]{fontenc}

\theoremstyle{plain}
\newtheorem{theorem}{Theorem}[section]

\newtheorem{lemma}[theorem]{Lemma}

\newtheorem{observation}[theorem]{Observation}

\newtheorem{conjecture}[theorem]{Conjecture}

\theoremstyle{definition}

\newtheorem{question}[theorem]{Question}

\newtheorem{def/prop}[theorem]{Definition/Proposition}

\pgfplotsset{compat = 1.16}

\DeclareMathOperator{\Pop}{\mathrm{Pop}}
\DeclareMathOperator{\Id}{\mathrm{Id}}
\DeclareMathOperator{\Inv}{\mathrm{Inv}}

\DeclareMathOperator{\Prob}{\mathbb{P}}

\DeclareMathOperator{\eps}{\varepsilon}

\newcommand{\cF}{\mathcal{F}}

\newcommand{\cS}{\mathcal{S}}

\title{\scshape Lower bound on the running time of Pop-Stack Sorting on a random permutation}

\author{Lyuben Lichev}
\affil{Department of Mathematics, Univ. Jean Monnet, Saint-Etienne, France}

\begin{document}

\maketitle

\begin{abstract}
Pop-Stack Sorting is an algorithm that takes a permutation as an input and sorts its elements. It consists of several steps. At one step, the algorithm reads the permutation it has to process from left to right and reverses each of its maximal decreasing subsequences of consecutive elements. It terminates at the first step that outputs the identity permutation.

In this note, we answer a question of Defant on the running time of Pop-Stack Sorting on the uniform random permutation $\sigma_n$. More precisely, we show that there is a constant $c > 0.5$ such that asymptotically almost surely, the algorithm needs at least $cn$ steps to terminate on $\sigma_n$.
\end{abstract}
\noindent
MSC Class: 60C05, 68Q87

\section{Introduction}

Given a positive integer $n$, we denote by $\cS_n$ the set of permutations of the integers from 1 to $n$. Here, we see the permutations in $\cS_n$ as $n$-tuples where every integer in $[n]$ is met exactly once. The identity permutation in $\cS_n$ is denoted by $\Id_n$, and for every $k\in [n]$ and a permutation $\sigma\in \cS_n$, the position of $k$ in $\sigma$ is denoted $\sigma^{-1}(k)$. Moreover, the \emph{distance} between two positions $\ell,m\in [n]$ is $|\ell-m|$.

Pop-Stack Sorting is an algorithm that takes a permutation as an input and sorts its elements. It consists of several \emph{steps} which transform the input permutation into the identity as follows. A \emph{run} in a permutation is a maximal decreasing subsequence of consecutive entries. At every step, the algorithm reads the permutation it has to process and reverses each of its runs. For example, one step of the algorithm on the permutation $(5,2,4,6,3,1)$ produces $(2,5,4,1,3,6)$, the next step produces $(2,1,4,5,3,6)$ followed by $(1,2,4,3,5,6)$ and $\Id_6$ (which is the unique fixed point in $\cS_6$ of the algorithm). For a permutation $\pi$, we denote $\Pop(\pi)$ the result of one step of the algorithm on $\pi$. 

Pop-Stack Sorting was introduced by Avis and Newborn~\cite{AN81} as a variant of Knuth's stack-sorting machine. In practice, to execute the algorithm, one only needs a `handicapped' stack in which elements may be pushed one at a time but once they are popped, all of them must come out at once and in reversed order. Then, a step of the algorithm comes down to reading the given permutation from left to right and for every $i\in [n]$, when the $i$-th position is reached:
\begin{itemize}
    \item if $\sigma(i)$ is larger than the head of the stack, push $\sigma(i)$ in the stack,
    \item otherwise, pop all elements in the stack first, and then push $\sigma(i)$ in the stack.
\end{itemize}
The analysis of this and related permutation sorting algorithms attracted a lot of attention in the last few years, see~\cite{ABBHL19, ABH21, CG19, CGP21, Def21a, EG19, PS18}. Moreover, the observation that $\Pop$ may be defined in wider generality on meet-semilattices led to a series of papers analyzing Pop-Stack Sorting on Tamari lattices~\cite{Hon22}, $\nu$-Tamari lattices~\cite{Def22a}, type A crystal lattices~\cite{DW22} and other lattices~\cite{CS22}, as well as Coxeter groups~\cite{Def21b}. For a detailed account on stack-sorting algorithms, we invite the reader to consult Defant's PhD thesis~\cite{Def22PhD}.

In this short note, we will be interested in the number of steps of the algorithm on a uniform random permutation $\sigma_n\in \cS_n$ as $n\to \infty$. The following conjecture was raised by Defant (see~\cite{Def21a}, Conjecture~4.3 and~\cite{Def22PhD}, Conjecture~10.32) and reiterated at the Banff workshop on Analytic and Probabilistic Combinatorics in November 2022. 

\begin{conjecture}\label{conj 1}
A.a.s.\ Pop-Stack Sorting needs $(1-o(1))n$ steps to reach $\Id_n$ from $\sigma_n$.
\end{conjecture}

On the one hand, Ungar~\cite{Ung82} showed that for every $n\in \mathbb N$ and every permutation $\pi\in \cS_n$, Pop-Stack Sorting terminates on $\pi$ after at most $n-1$ steps. Thus, if Conjecture~\ref{conj 1} is true, one has to come up with an asymptotically matching lower bound. On the other hand, it is easy to show that a.a.s.\ the algorithm needs at least $(1/2-o(1))n$ steps to reach $\Id_n$ from $\sigma_n$, which was already known by Defant~\cite{Def21a, Def22PhD} and follows directly by combining Observations~\ref{ob:1/2},~\ref{ob:1/2+} and Lemma~\ref{lem:1/2+}.

During his talk at Banff, Defant asked if the following weaker version of Conjecture~\ref{conj 1} holds, see also Section~4.4 in~\cite{Def21a} and Section~10.3.4 in~\cite{Def22PhD}.

\begin{question}\label{pb 1}
Is it true that there is $c > 0.5$ such that a.a.s.\ Pop-Stack Sorting needs at least $c n$ steps to reach the identity permutation from $\sigma_n$?
\end{question}

Our goal in this note is to answer Question~\ref{pb 1} in the affirmative. This requires a careful analysis of the algorithm that is non-trivial for several reasons. First of all, for every $\pi\in \cS_n$, define $\Inv(\pi)$ as the set of pairs $i,j\in [n], i<j$ such that $\pi^{-1}(j) < \pi^{-1}(i)$. Then, a natural partial order one may define on $\cS_n$ is $\sigma \le \pi$ if $\Inv(\sigma) \subseteq \Inv(\pi)$ (also known as the \emph{left weak order}). However, unlike many classical sorting algorithms (as, for example, Bubble Sorting), $\Pop$ is not monotone with respect to $\le$ in the sense that $\sigma \le \pi$ does not imply that $\Pop(\sigma)\le \Pop(\pi)$. For example, while in $(n, n-1, \ldots, 2, 1)$ (which is the unique maximal element with respect to $\le$ in $\cS_n$) all pairs of elements are inverted, it takes only one step to transform this permutation to $\Id_n$ (which is the unique minimal element with respect to $\le$ in $\cS_n$). Another point is that the classical inductive scheme that deletes the maximal or the minimal element in a permutation (which is helpful in a wide variety of combinatorial problems) is not suitable to analyze the number of steps of the algorithm. For example, consider the permutation $\pi = (4,1,3,5,2)$. While $\Pop^2(\pi) = \Id_5$, the permutation $(4,1,3,2)$ needs three steps to reach $\Id_4$. The key point here is that while the largest element (in this case 5) blocks certain exchanges of elements in $\pi$ (the one between 3 and 2), this may lead to longer runs later on in the process (here $4, 3, 2$). We believe (but do not know how to prove) that in pathological examples as the one above, deleting the maximal element can only accelerate the process by a bounded number of steps irrespectively of the length of the permutation.

The following theorem is our main result.

\begin{theorem}\label{thm 1}
A.a.s.\ Pop-Stack Sorting needs at least $0.503 n$ steps to reach the identity permutation from $\sigma_n$.
\end{theorem}

The main idea in the proof of Theorem~\ref{thm 1} is roughly as follows. First of all, setting $N = n^{2/3}$, we show that there is an integer $k\in [n-N+1, n]$ in some of the first $N$ positions in $\sigma_n$. We use this to show that for every $c, \eps > 0$ such that $2c+\eps < 1$,  a.a.s.\ the position of $k$ in $\Pop^{cn}(\sigma_n)$ is still at most $(1-\eps)n$. At the same time, we show that for a suitable choice of $c$ and $\eps$, a.a.s.\ for every $t\ge cn$, the last $\eps n$ integers in $\Pop^t(\sigma_n)$ may be partitioned into an increasing subsequence with terms that are not necessarily consecutive, and a set of numbers with positions at distance at least 2 from each other that move by one position to the right at step $t+1$. Finally, we use the above structure to show that once $k$ reaches the last $(1-\eps)n$ positions, it starts jumping by exactly one position to the right at every step, and thus conclude that $k$ needs at least $(1/2+\eps/2-o(1))n$ steps to reach its position in $\Id_n$. 

\paragraph{Notation and terminology.} In this paper, we skip lower and upper integer parts if rounding does not influence the argument for better readability. 

Given a sequence of probability spaces $(\Omega_n, \cF_n, \mathbb P_n)$ and events $E_n\in \cF_n$ for all $n\ge 1$, we say that $(E_n)_{n\ge 1}$ holds \emph{asymptotically almost surely}, or \emph{a.a.s.}\ for short, if $\mathbb P_n(E_n)\to 1$ as $n\to \infty$. However, we often abuse notation and say that $E_n$ itself holds a.a.s.

For integers $k,\ell\in [n]$, we say that $k$ \emph{jumps over} $\ell$ at step $t$ if the order of the positions of $k$ and $\ell$ changes at the $t$-th step of the algorithm. Finally, we say that an event happens \emph{at time $t$} if it happens between steps $t-1$ and $t$.

\section{Preliminaries on random permutations}

In this section, we gather some results on random permutations used in the proof of Theorem~\ref{thm 1}. 

\begin{observation}\label{ob:1/2}
Define $N = n^{2/3}$. Then, a.a.s.\ there is $k\in [n-N+1,n]$ such that $|\sigma_n^{-1}(k) - k|\ge n-2N$.
\end{observation}
\begin{proof}
This follows from the fact that the probability that none of the first $N$ positions in $\sigma_n$ is occupied by a number in the interval $[n-N+1, n]$ is $\prod_{i=0}^{N-1} (1 - \tfrac{N}{n-i})\le (1-\tfrac{N}{n})^N = o(1)$.
\end{proof}

\begin{observation}\label{ob:1/2+}
A.a.s.\ for every $k\in [n]$, the positions of $k$ in $\sigma_n$ and in $\Pop(\sigma_n)$ are at distance at most $2\log_2 n$ from each other.
\end{observation}
\begin{proof}
Set $s = 2\log_2 n$. Then, an elementary first moment computation shows that the expected number of runs of length $s$ in $\sigma_n$ is $(n-s+1)/s!\le n/2^{s-1} = o(1)$. Then, by Markov's inequality a.a.s.\ there is no such run in $\sigma_n$, which proves the observation.
\end{proof}

We continue with a couple of concentration lemmas. The proofs consist of standard second moment arguments.

\begin{lemma}\label{lem:conc}
Fix $b\in (0,1)$ and denote by $X$ the number of integers in the interval $[bn+1, n]$ whose positions in $\sigma_n$ are also in the interval $[bn+1, n]$. Then, $\mathbb E X = (1-b)^2 n$ and  for every $\delta > 0$,
\[\Prob(|X-\mathbb E X|\ge \delta n) \xrightarrow[n\to \infty]{} 0.\]
\end{lemma}
\begin{proof}
On the one hand, for every position $k\in [bn+1, n]$, there is probability $1-b$ that $\sigma(k)\ge bn+1$. Thus, $\mathbb E X = (1-b)^2 n$. On the other hand, 
\begin{align*}
& \mathrm{Var}(X) = \sum_{i=bn+1}^n \Prob(\sigma_n(i)\in [bn+1,n])  - \Prob(\sigma_n(i)\in [bn+1,n])^2\\ 
+\; 
& \sum_{i\neq j} \Prob(\{\sigma_n(i)\in [bn+1,n]\}\cap \{\sigma_n(j)\in [bn+1,n]\}) - \Prob(\sigma_n(i)\in [bn+1,n])\Prob(\sigma_n(j)\in [bn+1,n]),
\end{align*}
where the last sum goes over all pairs of distinct positions in $[bn+1, n]$. Thus, while the first sum is equal to $n((1-b)-(1-b)^2) = O(n)$, the second sum is at most $n^2\cdot \left(\tfrac{(1-b)n\cdot ((1-b)n-1)}{n(n-1)} - \left(\tfrac{(1-b)n}{n}\right)^2\right) = o(n^2)$. Finally, Chebyshev's inequality implies that for every $\delta > 0$,
\[\Prob(|X-\mathbb E X|\ge \delta n)\le \frac{\mathrm{Var}(X)}{\delta^2 n^2} = o(1),\]
as desired.
\end{proof}

\begin{lemma}\label{lem:Y}
Fix $\eps\in (0,1)$ and denote by $Y$ the number of positions $i\in [n]$ such that the number of integers larger than $\sigma_n(i)$ in a position in the interval $[i+1,n]$ is at most $\eps n$. Then, 
\[\mathbb E Y = (\eps \log(\eps^{-1})+\eps+o(1))n,\] 
and for every $\delta > 0$,
\[\Prob(|Y-\mathbb E Y|\ge \delta n) \xrightarrow[n\to \infty]{} 0.\]
\end{lemma}
\begin{proof}
Define $N = n^{2/3}$ and for every $i\in [n]$, define $\ell_i$ as the number of positions $k\in [i+1,n]$ such that $\sigma_n(i) < \sigma_n(k)$. Also, let $(U_i)_{i=1}^n$ be a sequence of $n$ independent random variables, each distributed uniformly in the interval $[0,1]$. We construct a map $\pi_n:[n]\to [n]$ by setting $\pi_n(i) = |\{j\in [n]: U_j \le U_i\}|$ for every $i\in [n]$. Then, $\pi_n$ a.s.\ belongs to $\cS_n$, and on this event, it is distributed as $\sigma_n$.

We show the lemma for $\pi_n$; for simplicity of notation, we keep using $Y$ and $(\ell_i)_{i=1}^n$. Define $Z$ as the number of positions $i\in [(1-\eps)n]$ such that $\ell_i\le \eps n$. Then,
\[\mathbb E Z = \sum_{i=1}^{(1-\eps)n} \Prob(\ell_i\le \eps n) = \sum_{i=1}^{(1-\eps)n} \int_{0}^1 \Prob(|\{j\in [i+1,n]: U_j > t\}|\le \eps n) \mathrm{d}t.\]
However, $|\{j\in [i+1,n]: U_j > t\}|$ is distributed as a Binomial random variable $\mathrm{Bin}(n-i, 1-t)$. Thus, an immediate application of Chernoff's inequality shows that if $(n-i)(1-t) - \eps n \ge N$ (or equivalently $t\le 1 - \tfrac{\eps n + N}{n-i}$), the above probability is $o(1)$, while if $(n-i)(1-t) - \eps n \le N$ (or equivalently $t\ge 1 - \tfrac{\eps n - N}{n-i}$), the same probability is $1-o(1)$. Since $N = o(n)$, we have that
\[\mathbb E Z = \sum_{i=1}^{(1-\eps)n} \left(1 - \left(1 - \frac{\eps n}{n-i} + o(1)\right)\right) = \left(\int_{\eps}^1 \frac{\eps}{u}\mathrm{d}u + o(1)\right)n = (- \eps \log\eps + o(1))n,\]
where the second equality comes from a classical integral approximation of the sum $\sum_{i=1}^{(1-\eps)n} \frac{n}{n-i}$. This shows the first point as $\eps \log(\eps^{-1}) = -\eps\log \eps$ and $\mathbb E Y = \eps n + \mathbb E Z$.

Moreover,
\begin{align*}
\mathrm{Var}(Z)\le \mathbb E Z + \sum_{i\neq j} \Prob(\{\ell_i\le \eps n\}\cap \{\ell_j \le \eps n\}) - \Prob(\ell_i \le \eps n)\Prob(\ell_j \le \eps n).
\end{align*}

However, by independence of $U_i$ and $U_j$ and a similar argument as the one above, we have that $ \Prob(\{\ell_i\le \eps n\}\cap \{\ell_j \le \eps n\})$ is bounded from below by
\[(1-o(1)) \Prob(U_i \ge 1-\tfrac{\eps n - N}{n-i}) \Prob(U_j \ge 1-\tfrac{\eps n - N}{n-j}) = (1-o(1)) \Prob(\ell_i \le \eps n)\Prob(\ell_j \le \eps n),\]
and bounded from above by
\[(1+o(1)) \Prob(U_i \ge 1-\tfrac{\eps n + N}{n-i}) \Prob(U_j \ge 1-\tfrac{\eps n + N}{n-j}) = (1+o(1)) \Prob(\ell_i \le \eps n)\Prob(\ell_j \le \eps n).\]
We conclude that
\begin{align*}
\mathrm{Var}(Z)\le \mathbb E Z + \sum_{i\neq j} o(\Prob(\ell_i \le \eps n)\Prob(\ell_j \le \eps n)) = o(n^2).
\end{align*}
Finally, Chebyshev's inequality implies that for every $\delta > 0$,
\[\Prob(|Y-\mathbb E Y|\ge \delta n)\le \frac{\mathrm{Var}(Y)}{\delta^2 n^2} = o(1),\]
as desired.
\end{proof}

\section{Proof of Theorem~\ref{thm 1}}

We begin with an easy observation about the structure of the permutations in the image of $\Pop$.

\begin{lemma}\label{lem:1/2+}
For every $\sigma\in \cS_n$, the longest run in $\Pop(\sigma)$ is of length at most $3$. 
\end{lemma}
\begin{proof}
Suppose for contradiction that there is $\sigma\in \cS_n$ such that for $\pi = \Pop(\sigma)$, there is $k\in [n-3]$ such that $\pi^{-1}(k) > \pi^{-1}(k+1) > \pi^{-1}(k+2) > \pi^{-1}(k+3)$. Then, for both $i\in \{k+1,k+2\}$, $\sigma^{-1}(i) = \pi^{-1}(i)$; indeed, every position that gets updated in $\pi$ is part of an increasing subsequence of at least two consecutive elements. However, this means that $\sigma^{-1}(k+1) > \sigma^{-1}(k+2)$, and therefore $\pi^{-1}(k+1)$ and $\pi^{-1}(k+2)$ must form an increasing subsequence. This contradiction shows the lemma.
\end{proof}

In the sequel, for every $t\ge 1$, we denote $\sigma_n^t = \Pop^t(\sigma_n)$. Also, we define $L_t$ as the set of positive integers $s$ such that every integer on the right of them in $\sigma_n^t$ is larger than $s$ (which are often called \emph{right-to-left minima} of $\sigma_n^t$).

\begin{lemma}\label{lem:L_k}
For every $b\in (0,1)$ and every integer $s\in [(1-b/4)n+1, n]$ such that $\sigma_n^{-1}(s)\ge (1-b/4) n+1$, $s$ belongs to $L_{bn}$.
\end{lemma}
\begin{proof}
Fix $b$ and $s$ as above. To begin with, $\sigma_n^{-1}(s)$ cannot reach a position smaller than $(1-b/2) n + 1$ in any of $(\sigma_n^t)_{t\ge 1}$ since at most $bn/4-1$ integers can jump over $\sigma_n^{-1}(s)$ from left to right. Now, for every $t\ge 0$, define $R_t$ as the set of integers that are smaller than $s$ and on the right of $s$ in $\sigma_n^t$. 

We define an auxiliary process (which is a deterministic version of the asymetric simple exclusion process) as follows. In the beginning, put particles on each of the largest $bn/2$ integers in $[n]$. We say that a position with a particle is \emph{occupied} while a position without a particle is a \emph{hole}. At the first step, we stay with the same configuration of particles. At each of the next steps, for every pair of consecutive positions such that the first is a hole and the second is occupied, move the particle in the occupied position to the hole (that is, one position to the left). Then, for every $t\ge 0$, denote by $R'_t$ the set of positions among $[(1-b/2)n+1, n]$ that are occupied after $t$ steps. 

Now, for two sets $A,B\subseteq [n]$, we say that $A \le B$ if $|A|\le |B|$ and for every $r\in [|A|]$, the $r$-th largest element in $A$ is smaller than or equal to the $r$-th largest element in $B$. This defines a partial order on the family of subsets of $[n]$. 

For every $t\ge 0$ and $r\in [|R_t|]$, denote by $\ell_{r,t}$ the $r$-th largest element in $R_t$. Note that for every fixed $r$, $\ell_{r,t}$ is decreasing in $t$ because the sequence of sets $(R_t)_{t\ge 0}$ is itself decreasing for the order relation on the subsets of $[n]$. 

We show by induction that for every $t\ge 0$, $R_t \le R'_t$, or equivalently that for all $t\ge 0$ and $r\in [|R_t|]$, the positions $[\ell_{t,r}, n]$ contain at least $r$ particles of $R_t'$. The cases $t=0$ and $t=1$ are clear. Suppose that $R_{t-1} \le R'_{t-1}$ for some $t\ge 2$, and fix any $r\in [|R_{t-1}|]$. Then, if $|[\ell_{r,t-1}+1, n]\cap R_{t-1}'|\ge r$, one must have that $|[\ell_{r,t}, n]\cap R_t'|\ge r$ because $\ell_{r,t}\le \ell_{r,t-1}$ and every particle in the auxiliary process can move by at most one position to the left at step $t$. Suppose that $|[\ell_{r,t-1}+1, n]\cap R_{t-1}'| = r-1$, that is, $\ell_{r,t-1}$ is the $r$-th largest position in $R_{t-1}'$. We consider two cases. On the one hand, if $(\sigma_n^{t-1})^{-1}(\ell_{r,t-1}-1) < s$, then by the induction hypothesis $R_{t-1}'$ contains position $\ell_{r,t-1} - 1$, and therefore $\ell_{r,t-1}$ remains the $r$-th largest position in $R_t'$. Hence, using that $\ell_{r,t}\le \ell_{r,t-1}$ shows that $|[\ell_{r,t}, n]\cap R_t'|\ge r$. On the other hand, if $(\sigma_n^{t-1})^{-1}(\ell_{r,t-1}-1) \ge s$, then either $(\sigma_n^{t-1})^{-1}(\ell_{r,t-1}+1) < (\sigma_n^{t-1})^{-1}(\ell_{r,t-1})$, in which case the numbers in positions $\ell_{r,t-1}-1$ and $\ell_{r,t-1}+1$ exchange their positions at step $t$, or $(\sigma_n^{t-1})^{-1}(\ell_{r,t-1}+1) > (\sigma_n^{t-1})^{-1}(\ell_{r,t-1})$, in which case the numbers in positions $\ell_{r,t-1}-1$ and $\ell_{r,t-1}$ exchange their positions at step $t$. In both cases, the position of the $r$-th integer smaller than $s$ becomes $\ell_{r,t-1}-1$, which is equal to the position of the $r$-th particle in the auxiliary process after step $t$.



Finally, it is not hard to see that the auxiliary process needs $1+2\cdot bn/2 - 1 = bn$ steps to evacuate all particles from the interval $[(1-b/2)n+1, n]$ (where the first 1 comes from the fact that the auxiliary process does not evolve at all at the first step), which implies that $R_{bn} = R'_{bn} = \emptyset$, as desired.
\end{proof}




Now, fix $a\in (0,1)$ to be suitably chosen later. Then, by Lemma~\ref{lem:conc} a.a.s.\ there are $(a^2/16-o(1)) n$ integers $k\in [(1-a/4)n+1, n]$ such that $\sigma_n^{-1}(k)\in [(1-a/4)n+1, n]$. Denote by $m$ the smallest of these integers. At step $an$ of the Pop-Stack Sorting algorithm, we color the numbers larger or equal to $m$ in three colors that evolve dynamically. More precisely, for any $t\ge an$, a number $k\ge m$ is \emph{white} if it belongs to $L_t$ (in which case $(\sigma_n^t)^{-1}(m) \le  (\sigma_n^t)^{-1}(k)$), it is \emph{red} if $(\sigma_n^t)^{-1}(m) \le  (\sigma_n^t)^{-1}(k)$ but $k\notin L_t$, and \emph{black} otherwise. In particular, a number may be recolored from red to white, and from black to red or white. 

For every $t\ge an$, we denote by $p_t$ the largest position such that $\sigma_n^t$ contains two consecutive red numbers in positions $p_t - 1$ and $p_t$. If such a pair does not exist, we set $p_t = (\sigma_n^t)^{-1}(m)$. 

\begin{lemma}\label{lem:reds}
For every $t\ge an$, $p_{t+2}\le p_t$.
\end{lemma}
\begin{proof}
Fix $t\ge an$. To begin with, every red number in position larger than $p_t$ in $\sigma_n^t$ is surrounded by white numbers, so at the next step it moves by one position to the right and remains surrounded by white numbers. We concerntrate on the case $p_t > (\sigma_n^t)^{-1}(m)$, as the case $p_t = (\sigma_n^t)^{-1}(m)$ follows by a similar case analysis.

There are two possibilities for the red number at position $p_t$. If the red number on its left is smaller than itself, $(\sigma_n^t)^{-1}(p_t)$ makes one step to the right and hence $p_{t+1}\le p_t - 1$. Thus, using that $(\sigma_n^{t+1})^{-1}(p_t)$ is white in this case, Lemma~\ref{lem:1/2+} implies that $p_{t+2}\le p_{t+1}+1\le p_t$. If the red number in position $p_t-1$ is larger than the one in position $p_t$, $(\sigma_n^t)^{-1}(p_t-1)$ moves two positions to the right at step $t+1$ so $p_{t+1} = p_t + 1$. Then, $(\sigma_n^{t+1})^{-1}(p_t) < (\sigma_n^{t+1})^{-1}(p_t+1)$ and $(\sigma_n^{t+1})^{-1}(p_t+2)$ is white, so $p_{t+2} = p_{t+1}-1 = p_t$, which finishes the proof.
\end{proof}

\begin{lemma}\label{lem:queue}
Fix $a\in (0, 0.5)$, $c\in (a, 0.5)$ and $\eps > 0$ such that $\eps \log(\eps^{-1})+2\eps < \tfrac{c-a}{2}$ and $\eps < \tfrac{a^2}{16}$. Then, a.a.s.\ $p_{cn} \le (1-\eps)n$.
\end{lemma}
\begin{proof}
Consider the sequence $(p_{an+2u})_{u=0}^{(c-a)n/2}$ starting with $p_{an}$ and ending with $p_{cn}$. By Lemma~\ref{lem:reds} this sequence is decreasing. Suppose that $p_{cn} > (1-\eps)n$. Then, there are at least $(\tfrac{c-a}{2} - \eps)n$ steps $t\in [an, cn-2]$ such that $p_t = p_{t+2}$.

On the other hand, from the case analysis in the proof of Lemma~\ref{lem:reds} it follows that a.a.s.\ for every $t$ such that $p_t = p_{t+2}$, there is an integer that jumped from left to right between $p_t-1$ and $p_t$ at step $t$ or $t+1$. Moreover,  a.a.s.\ this integer is red since a.a.s.\ $p_t > (1-\tfrac{a^2}{16}+o(1))n\ge (\sigma_n^t)^{-1}(m)$, and in position $p_t - 1$ or $p_t$ at time $t+1$. Since every integer can do this at most once (recall that red integers with position in the interval $[p_t+1, n]$ at time $t+1$ are surrounded by white integers, and thus move one position to the right), this means that at least $(\tfrac{c-a}{2} - \eps)n$ integers visited the interval $[(1-\eps)n+1, n]$ at some point in the algorithm. In particular, with the notation of Lemma~\ref{lem:Y}, $Y\ge (\tfrac{c-a}{2} - \eps)n$. Choosing $\eps$ such that $\eps \log(\eps^{-1}) + 2\eps < \tfrac{c-a}{2}$ and using Lemma~\ref{lem:Y} shows that a.a.s.\ the above event does not happen and proves the lemma.
\end{proof}

Now, we are ready to show Theorem~\ref{thm 1}.

\begin{proof}[Proof of Theorem~\ref{thm 1}]
Define $N$ as in Observation~\ref{ob:1/2} and set $a = 0.32$, $c = 0.48$ and $\eps_0 = 0.0061$. Then, we have that after $cn$ steps, a.a.s.\ the position of the number $N$ is at most $(2c+o(1))n$. 

On the other hand, $\eps_0 \log(\eps_0^{-1}) + 2\eps_0 < \tfrac{c-a}{2}$ and $\eps_0 < a^2/16$, and therefore Lemma~\ref{lem:queue} implies that $p_{cn}\le (1-\eps_0)n$. Thus, since $2c + \eps_0 < 1$, the number $N$ must move by only one position to the right during its last $\eps_0 n - N$ moves (as do all red integers once they jump over $p_t$ for some $t\ge cn$) before reaching its position in $\Id_n$. Thus, in total, $N$ needs at least $(1/2+\tfrac{\eps_0}{2} - o(1))n$ steps to get in place. Using that $\tfrac{\eps_0}{2} > 0.003$ shows the theorem. 
\end{proof}

\section{Concluding remarks}

In this short note, we answer positively to an open question on the running time of the Pop-Stack Sorting algorithm on a random permutation. The proof is based on a precise analysis of the algorithm on a segment containing the last $\eps n$ positions for some small (but fixed) value of $\eps$. While the approach may be optimized further to ensure a slightly larger constant, resolving Conjecture~\ref{conj 1} remains currently out of reach. Indeed, if the conjecture is true, a more thorough treatment of the early steps of the process will be needed to confirm it. Our work does not provide further insights on this point: note that our understanding of the early steps of the process does not surpass the (purely combinatorial) Lemma~\ref{lem:L_k}.

\paragraph{Acknowledgements.} The author is grateful to Marcos Kiwi and Dieter Mitsche for sharing the problem. Thanks are also due to Colin Defant and Dieter Mitsche for several useful remarks and suggestions.



\bibliography{Refs}
\bibliographystyle{plain}






\end{document}